\documentclass{article}
\usepackage{amssymb}
\usepackage{amsthm}
\usepackage{amsfonts}
\usepackage[centertags]{amsmath}

\newtheorem{theorem}{Theorem}

\newtheorem{corollary}[theorem]{Corollary}

\newtheorem{definition}[theorem]{Definition}
\newtheorem{example}[theorem]{Example}

\newtheorem{lemma}[theorem]{Lemma}

\newtheorem{proposition}[theorem]{Proposition}
\newtheorem{remark}[theorem]{Remark}

\begin{document}

\date{}
\author{Thabet ABDELJAWAD\footnote{\c{C}ankaya University, Department of
Mathematics, 06530, Ankara, Turkey} }
\title{On Some Topological Concepts of TVS-Cone Metric Spaces and Fixed Point Theory Remarks}
\maketitle

\begin{abstract}
We prove that every TVS-cone metric space (i.e a cone metric space over a locally convex topological vector  space $E$) is first countable paracompact topological  space and by using Du's results in " [A note on cone metric fixed point theory and its equivalence, {Nonlinear Analysis},72(5),2259-2261 (2010)]", we conclude that every TVS-cone metric space is topologically isomorphic to a topological metric space. We also show how to construct comparable metric topologies to TVS-cone metric topologies by using the system of seminorms generating the topology of the locally convex topological vector space $E$. When $E$ is a Banach space these metric topologies turn to be equivalent to the original TVS-cone metric topologies. Even though, we remark that there are still some fixed point theorems to deal nontrivially with them in TVS-cone metric spaces. The nonlinear scalarization is used also to prove some fixed point theorems with nonlinear contractive conditions.
\end{abstract}


\section{Introduction and Preliminaries}
The cone metric space fixed point theory was initiated in 2007,  by Huang and Zhang \cite{HZ}. Afterwards, many authors started to generalize fixed point theorems from metric spaces to cone metric spaces (CMS) (see \cite{RH},\cite{Ishak},\cite{TAA},\cite{TAA2},\cite{K},\cite{T},\cite{AK}, \cite{TA}). Through those generalizations, normal, sequentially regular and strongly minihedral cones were crucial. However, some results have been carried to CMS without any extra assumptions on the cone. Some other authors cared about the topological structures in CMS ( \cite{TA} and \cite{Ayse}). It was proved that CMS are first countable paracompact topological spaces provided that the cone must not have empty interior.

Recently, Du \cite{D_2009} gave the definition of a more generalized cone
metric space, namely topological vector space-cone metric space
(TVS-CMS) and   showed that Banach contraction principle in
usual metric spaces and in TVS-CMS are equivalent. Actually, he concluded that certain contractive type conditions can be transformed into equivalent ones in the correspondent usual metric space by applying the nonlinear scalarization. The fixed point theorems (for example Kannan's type and Chatterjea's type) containing such kinds of conditions have been generalized before to CMS without any normality or regularity assumption on the cone under consideration.

In this manuscript, we prove that every TVS-CMS is a first countable paracompact topological space without the need of normality assumption (see \cite{Ayse}). The proofs based on that every TVS-CMS is topologically isomorphic to a usual metric space. The key in defining these metric equivalent was the nonlinear scalarization function via some interior point of the cone \cite{D_2009}. However, we also describe how to construct such comparable metric topologies to the TVS-cone metric topologies by using the system of the seminorms generating the topology of the locally convex topological space $E$. Actually, when $E$ is Banach we arrive at an equivalent topology. Even though, we remark that  there are still some fixed point theorems that can not be proved trivially in TVS-CMS by making use of the equivalent metric metric topology. In the trivialization direction, we showed also how to transform some nonlinear contractive type conditions from TVS-CMS to usual metric spaces.

Throughout this paper, $(E,S)$ stands for real Hausdorff locally convex topological vector space (t.v.s.) with $S$ its generating system of seminorms.
A non-empty subset $P$ of $E$ is called  cone if $P+P \subset P$, $\lambda P \subset P$
for $\lambda \geq 0$ and  $P \cap (-P) =\{0\}$. The cone $P$ will be assumed to be closed with nonempty interior as well, that is, $P$ will be a convex pointed cone with nonempty interior. The positive cones in $C[0,1]$ and $L_\infty [0,1]$ have nonempty interiors. However, the positive cone of $L_p[0,1],~0\leq p <\infty$ have empty interiors \cite{Ali}. Moreover, it is known that a real locally convex topological vector space has a continuous norm if and only if it contains a closed pointed convex cone with nonempty interior \cite{Helge}. There are many examples of Fr\'{e}chet K\"{o}the sequence spaces which admit continuous norms and hence contain closed convex pointed cones with nonempty interiors.
For a given cone $P$, one can define  a partial ordering (denoted by $\leq:$ or $\leq_P$) with respect to $P$
by $x\leq y$ if and only if $y-x \in P$. The notation $x<y$ indicates that $x\leq y$ and $x\neq y$ while
$x\ll y$ will show $y-x\in intP$, where $intP$ denotes the interior of $P$.
Continuity of the algebric operations in a topological vector space and the properties of the cone imply the relations:
$$intP+intP\subseteq intP ~\emph{and}~\lambda intP \subseteq intP~(\lambda > 0).$$
We appeal to these inclusions  in the following.

\begin{definition} \cite{Ali}
A cone $P$ of a topological vector space $(X,\tau)$ is said to be normal  whenever $\tau$ has a local base of zero consisting of $P-$ full sets, where a subset  $A$  is said to be $P-$full if   $\{a \in E: x\leq a \leq y\}\subset A$ for all $x, y \in A$.
\end{definition}
\begin{theorem} \cite{Ali}
(a) A cone $P$ of a topological vector space $(X,\tau)$ is normal  whenever $\{x_\alpha\}$ and $\{y_\alpha\}$, $\alpha \in \Delta$ are two nets in $X$ with $0\leq x_\alpha \leq y_\alpha$ for each $\alpha \in \Delta$ and $y_\alpha \rightarrow 0$, then $x_\alpha \rightarrow 0$.

(b) The cone of an ordered locally convex space $(X,\tau)$ is normal if and only if $\tau$ is generated by a family of monotone $\tau-$ continuous seminorms. Where a seminorm $q$ on $X$ is called monotone if $q(x)\leq q(y)$ for all $x, y \in X$ with $0\leq x \leq y$.
\end{theorem}

In particular, if   $P$ is a cone of a real Banach space $E$, then it  is called
\textit{normal} if there is a number $K \geq 1$ such that for
all $x,y \in E$:\  $
0\leq x \leq y\Rightarrow \|x\|\leq K \|y\|.$ The least positive integer $K$, satisfying this inequality,
is called the normal constant of $P$.
Also, $P$ is said to be  \textit{regular} if every increasing sequence which is bounded
from above is convergent. That is, if $\{x_n\}_{n\geq 1}$ is a sequence such
that $x_1 \leq x_2\leq \cdots\leq y$ for some $y \in E$, then there is $x \in E$
such that $\lim_{n\rightarrow\infty} \|x_n-x\|=0$. For more details about cones in locally convex topological vector spaces we may refer the reader to \cite{Ali}.

\begin{definition} \label{regular cone}\cite{De}
A cone $P$ of a locally convex topological vector space $(E,S)$ is said to be (sequentially) regular if every sequence in $E$ which is increasing and bounded above must be convergent in $(E,S)$. That is, if $\{a_n\}$ is a sequence in $E$ such that $a_1\leq a_2 \leq...\leq a_n \leq...\leq b$ for some $b \in E$, then there is $a \in E$ such that $q(a_n-a)\rightarrow 0$ as $n\rightarrow \infty$, for all $q \in S$.
 Equivalently, if every sequence in $E$ which is decreasing and bounded below is convergent.
\end{definition}
It is well-known that if $E$ is a real Banach space then every (sequentially) regular cone is normal.
\begin{definition} \label{strongly minihedral}\cite{De}
A cone $P$ of a locally convex topological vector space $(E,S)$ is said to be strongly minihedral if and only if any subset of $E$ which is bounded above must have a least upper bound. Equivalently, if any subset of $E$ which is bounded below must have a greatest lower bound.
\end{definition}
\begin{definition} (See \cite{CHY}, \cite{D_2008}, \cite{D_2009})
For $e \in intP$, the nonlinear scalarization function  $\xi_e:E\rightarrow \mathbb R$ is defined by
\[\xi_e(y)=\inf\{t \in \mathbb R: y \in te-P\}, \ \mbox{for all} \ y \in E.\]
\end{definition}
\begin{lemma} (See \cite{CHY}, \cite{D_2008}, \cite{D_2009})
For each $t\in \mathbb R$ and $y \in E$, the following are satisfied:
\begin{itemize}
\item[$(i)$] $\xi_e(y)\leq t\Leftrightarrow  y \in te-P$,
\item[$(ii)$] $\xi_e(y)> t\Leftrightarrow  y \notin te-P$,
\item[$(iii)$] $\xi_e(y)\geq t\Leftrightarrow  y \notin te-intP$,
\item[$(iv)$] $\xi_e(y)< t\Leftrightarrow  y \in te-intP$,
\item[$(v)$] $\xi_e(y)$ is positively homogeneous and continuous  on $E$,
\item[$(vi)$] if $y_1\in y_2+P$, then $\xi_e(y_2)\leq \xi_e(y_1)$,
\item[$(vii)$] $\xi_e(y_1+y_2)\leq \xi_e(y_1)+\xi_e(y_2)$, for all $y_1,y_2 \in E$.
\item[$(viii)$] if $y_1\in y_2+int(P)$, then $\xi_e(y_2)< \xi_e(y_1)$
\end{itemize}
\label{lemma_scalarization}
\end{lemma}
Lemma \ref{lemma_scalarization} (vi) and Lemma \ref{lemma_scalarization} (viii) mean that the nonlinear scalarization function $\xi_e$ is monotone and strictly monotone. However, it is not strongly monotone (see \cite{CHY}).

\begin{definition} Let $X$ be a non-empty set. and $E$ as usual a Hausdorff locally convex topological space. Suppose  a vector-valued function $p:X\times X\rightarrow E$ satisfies:
\begin{enumerate}
\item[$(M1)$] $0\leq p(x,y)$ for all $x,y \in X$,
\item[$(M2)$] $p(x,y)=0$ if and only if $x=y$,
\item[$(M3)$] $p(x,y)=p(y,x)$ for all $x,y \in X$
\item[$(M4)$] $p(x,y) \leq p(x,z)+p(z,y)$, for all $x,y,z \in X$.
\end{enumerate}
Then,  $p$ is called TVS-cone metric on $X$ and the pair $(X,p)$ is called
a  TVS-cone metric space (in short, TVS-CMS).
\end{definition}

Note that in \cite{HZ}, the authors considered $E$ as a real Banach space
in the definition of TVS-CMS. Thus, a cone metric space (in short, CMS) in the sense of Huang and
Zhang \cite{HZ} is a special case of TVS-CMS.

\begin{lemma} (See \cite{D_2009})
Let $(X,p)$ be a TVS-CMS. Then, $d_p:X \times X\rightarrow [0,\infty)$ defined by $d_p=\phi_e\circ p$ is a metric.
\label{lemma_usual_metric}
\end{lemma}

\begin{remark}
Since a cone metric space $(X,d)$ in the sense of Huang and Zhang \cite{HZ}, is a special case of TVS-CMS, then
$d_p:X \times X\rightarrow [0,\infty)$ defined by $d_p=\xi_e\circ d$ is also a metric.
\label{remark_CMS_usual_ms}
\end{remark}

\begin{definition}(See \cite{D_2009})
Let $(X,p)$ be a TVS-CMS, $x\in X$ and $\{x_n\}_{n=1}^{\infty}$ a sequence in $X$.
\label{definition_convergence}
\begin{itemize}
\item[($i$)] $\{x_n\}_{n=1}^{\infty}$ TVS-cone converges to $x\in X$ whenever
for every $0<<c\in E$, there is a natural number $M$ such that $p(x_n,x)<<c$
for all $n\geq M$ and denoted by $cone-\lim_{n\rightarrow \infty}x_n=x$ (or
$x_n\stackrel{cone}{\rightarrow} x$ as $n\rightarrow \infty$),
\item[($ii$)] $\{x_n\}_{n=1}^{\infty}$ TVS-cone Cauchy sequence in $(X,p)$ whenever
for every $0<<c\in E$, there is a natural number $M$ such that $p(x_n,x_m)<<c$
for all $n,m \geq M$,
\item[($iii$)]  $(X,p)$ is TVS-cone complete if every sequence TVS-cone Cauchy sequence  in $X$ is a TVS-cone
convergent.
\end{itemize}
\end{definition}
\begin{lemma} (See \cite{D_2009})
Let $(X,p)$ be a TVS-CMS, $x\in X$ and $\{x_n\}_{n=1}^{\infty}$ a sequence in $X$.
Set $d_p=\xi_e\circ p$. Then the following statements hold:
\begin{itemize}
\item[($i$)] If $\{x_n\}_{n=1}^{\infty}$ converges to $x$ in TVS-CMS $(X,p)$, then
$d_p(x_n,x)\rightarrow 0$ as $n\rightarrow \infty,$
\item[($ii$)] If $\{x_n\}_{n=1}^{\infty}$ is Cauchy sequence in  TVS-CMS $(X,p)$, then
$\{x_n\}_{n=1}^{\infty}$ is a Cauchy sequence (in usual sense) in $(X,d_p)$,
\item[($iii$)] If $(X,p)$ is complete TVS-CMS, then $(X,d_p)$
is a complete metric space.
\end{itemize}
\label{lemma_eq_statements}
\end{lemma}

\begin{remark} \label{notice that}
Note that the implications in (i) and (ii) of Lemma \ref{lemma_eq_statements} are also conversely true. Regarding (i) we prove that if $x_n\rightarrow x$ in $(X,d_p)$, then $x_n\rightarrow x$ in $(X,p)$. To this end, let $c\gg0$ be given, and find $q \in S$ and $\delta >0$ such that $q(b)<\delta$ implies that $b<< c$. This is possible since we can find a symmetric neighborhood of $c$, $c+\{b:q(b)<\delta\}\subset intP$. Since $\frac{e}{n}\rightarrow 0$ in $(E,S)$, where $e\gg 0$, find $\epsilon = \frac{1}{n_0}$ such that $\epsilon q(e)=q(\epsilon e)<\delta $ and hence $\epsilon e\ll c$. Now, find $n_0$ such that $d_p (x_n,x)=\xi_e \circ p (x_n,x)< \epsilon$ for all $n\geq n_0$. Hence, by Lemma \ref{lemma_scalarization} (iv) $p (x_n,x)\ll \epsilon e\ll c$ for all $n\geq n_0$. The proof of the converse of implication (ii) is similar. Now it is possible to say that the part (iii) of Lemma \ref{lemma_eq_statements} is an immediate consequence of (i) and (ii).
\end{remark}

\begin{proposition}(See \cite{D_2009})
Let $(X,p)$ is complete TVS-CMS and $T:X\rightarrow X$ satisfy the contractive
condition
\begin{equation}
p(Tx,Ty)\leq k p(x,y)
\label{contraction}
\end{equation}
for all $x,y \in X$ and $0 \leq k <1$. Then, $T$ has a unique fixed
point in $X$. Moreover, for each $x\in X$, the iterative sequence
$\{T^nx\}_{n=1}^{\infty}$ converges to fixed point.
\label{Du_thm22}
\end{proposition}
\begin{lemma} \label{th1}
Let $(X,d)$ be a cone metric space over a locally convex space $(E,S)$, where $S$ is the family of seminorms defining the locally convex topology. Let $\{x_n\}$ be a sequence in $E$. Then

(i) $x_n\rightarrow x$ in $(X,d)$ if and only if $d(x_n,x)\rightarrow 0$ in $(E,S)$.

(ii) $\{x_n\}$ is a Cauchy sequence in $(X,d)$    if and only if $\lim_{m,n\rightarrow \infty}d(x_n,x_m)=0$ in $(E,S)$.
\end{lemma}
\begin{proof}
(i) Suppose that $\{x_n\}$ converges to $x$. Let  $\varepsilon >0$ and $p \in S$ be given. Choose $c>>0$ such that $p(c)< \varepsilon$. This is possible by taking $c = \frac{\varepsilon c_0}{2p(c_0)}$, where $c_0$ is an interior point of $P$. Then there is $n_0$ such that $d(x_n,x)<<c$ for all $n>n_0$. Then by normality of the cone $P$ we have $p(d(x_n,x))\leq p(c)<\varepsilon$ for all $n>n_0$. This means $d(x_n,x)\rightarrow 0$ in $(E,S)$.
Conversely, suppose that $d(x_n,x)\rightarrow 0$ in $(E,S)$. For $c \in E$ with $c>>0$ find $\delta >0$ and $p \in S$ such that $p(b)< \delta$ implies $b<<c$. For this $\delta$ and this $p$ find $n_0$ such that $p(d(x_n,x))< \delta$ for all $n>n_0$ and so $d(x_n,x)<< c$ for all $n>n_0$. Therefore $x_n\rightarrow x$ in (X,d).

(ii) The proof is similar to that in (i).
\end{proof}

\begin{lemma} \label{th2}
Let $(X,d)$ be a TVS-cone metric space over a normal cone of a  locally convex space $(E,S)$, where $S$ is the family of seminorms defining the locally convex topology. Let $\{x_n\}$ and $\{y_n\}$ be two sequences in $X$ and $x_n\rightarrow x$, $y_n\rightarrow y$. Then $d(x_n,y_n)\rightarrow d(x,y)$ in $(E,S)$.
\end{lemma}
\begin{proof}
Let $\varepsilon >0$ and $p \in S$ be given. Choose $c \in E$ with $c>>0$ such that $p(c)< \frac{\varepsilon}{6}$. From $x_n\rightarrow x$ and  $y_n\rightarrow y$, find $n_0$ such that for all $n>n_0$, $d(x_n,x)<<c$ and $d(y_n,y)<<c$. Then for all $n>n_0$ we have
$$d(x_n,y_n)\leq d(x_n,x)+d(x,y)+ d(y,y_n)\leq d(x,y)+2c,$$

and
$$d(x,y)\leq d(x,x_n)+d(x_n,y_n)+ d(y_n,y)\leq d(x_n,y_n)+2c.$$
Hence
$$0\leq d(x,y)+2c-d(x_n,y_n)\leq 4c$$
 and so by the normality of $P$ we obtain
 $$p(d(x_n,y_n)-d(x,y))\leq p(d(x,y)+2c-d(x_n,y_n))+p(2c)\leq 6 p(c)< \varepsilon.$$ Therefore $d(x_n,y_n)\rightarrow d(x,y)$ in $(E,S)$.
\end{proof}
\section{Topological TVS-Cone Metric Spaces}

\begin{theorem} \label{TS}
Every TVS-CMS is a first countable topological space
\end{theorem}
\begin{proof}
Let $(X,\rho)$ be a TVS-CMS over a cone $P$ with nonempty interior. Let $\beta_\rho= \{ B(x,c): x \in X, c\gg 0\}$. If we show that for all $c_1\gg 0$, $c_2 \in E$ there exists a $c\gg 0$ such that $c\ll c_1$ and $c\ll c_2$ then we would be able to infer that $\tau_\rho=\{U\subseteq X: \forall x \in U~~\exists c\gg 0~\texttt{with}~B(x,c)\subseteq U\} \cup \{\Phi\}$ is a topology on $X$ with basis $\beta_\rho$. To this end let $c_1\gg 0$ and $c_2\gg 0$ be given and hence find $\delta > 0$ and $q \in S$ such that $q(b)<\delta$ implies $b\ll c_2$. Also, choose $n_0$ such that $\frac{1}{n_0}< \frac{\delta}{q(c_1)}$. Let $c=\frac{c_1}{n_0}$, then $c\gg 0$, $c\ll c_1$ and $c\ll c_2$. For the same interpretation $(X,\rho)$ is first countable. Actually, it can be shown that $\beta_x=\{B(x,\frac{c_0}{n}): n \in \mathbb{N}\}$ is a countable local base at $x$ for some $c_0\gg 0$, where $B(x,\frac{c_0}{n})=\{y \in X: \rho (x,y) \ll c_0\}$.
\end{proof}

\begin{corollary}\label{sc}
A mapping from a TVS-CMS to an arbitrary topological space is continuous if and only if it is sequentially continuous
\end{corollary}
The proof follows by Theorem \ref{TS}.
\begin{corollary}\label{isom}
Every  TVS-CMS $(X,p)$ is topologically isomorphic to its correspondent metric space $(X,d_p)$
\end{corollary}
The proof follows by  Remark \ref{notice that} and Corollary \ref{sc}.
\begin{corollary}\label{paracompact}
Every  TVS-CMS is a paracompact topological space and hence $T_4$ space.
\end{corollary}

The proof follows by Theorem \ref{TS}, Remark \ref{notice that} and that paracompactness is a topological property.

Corollary \ref{paracompact} above generalizes the result obtained recently by A. S\"{o}nmez in \cite{Ayse}. No normality assumption is assumed in our case.
\begin{corollary}\label{completions}
Every  TVS-CMS is  topologically isomorphic to a dense subspace of a complete metric space.
\end{corollary}
The proof follows by Corollary \ref{isom} and the completion theorem of usual metric spaces.

Note that in Corollary \ref{completions} above, we didn't say that the imbedding is isometrically isomorphic. This is because isometries are only defined between TVS cone metric spaces over the same $E$. However, for example, if $(X,p)$ is a CMS over a normal cone $P$ of a Banach $E$, then $(X,p)$ is isometrically isomorphic to a dense subspace of a complete CMS over the same cone $P$ ( see \cite{T}).
\begin{corollary}\label{compact}
A TVS-CMS is compact if and only if it is sequentially compact.
\end{corollary}
The proof follows by Corollary \ref{isom}, compactness and sequential compactness are topological properties and that a subset of a metric space is compact if and only if it is sequentially compact.

\begin{corollary}
Every TVS-CMS is of second category.
\end{corollary}
The proof followed by Theorem \ref{TS} and that being of second category is a topological property.
\begin{definition}\label{sclosed}
A subset $A$ of a TVS-CMS $(X,p)$ is called sequentially closed if whenever $x_n \in A$ with $x_n\rightarrow x$ then $x \in A$.
\end{definition}
\begin{corollary} \label{scl}
A subset $A$ of a TVS-CMS $(X,p)$ is  sequentially closed if and only if $A$ is closed.
\end {corollary}
The proof follows by Theorem \ref{TS}.
\begin{proposition} \label{scts}
Let $(X,p)$ be  a TVS-CMS. Then the ball $\overline{B}(x,c_0)=\{y \in X: p(x,y)\leq c_0\}$, $c_0\gg 0$, is sequentially closed and hence closed.
\end{proposition}
\begin{proof}
Let $y_n \in \overline{B}(x,c_0) $ be a sequence such that $y_n\rightarrow y$. Then, for each $m \in \mathbb{N}$ there exists $k_m$ such that $p(x,y)\leq p(x,y_{k_m})+p(y_{k_m},y)\ll c_0+ \frac{c_0}{m}$. Hence, $\frac{c_0}{m}+c_0-p(x,y) \in P$. Since $P$ is closed then we conclude that $c_0-p(x,y) \in P$ and so $p(y,x)\leq c_0$.
\end{proof}

The nonlinear scalarization is a way to generate metric topologies comparable to the TVS-cone metric topologies. The following theorem shows how to generate comparable metric topologies by using the generating system of seminorms.
\begin{theorem} \label{Stoplogy}
Let $(X,p)$ be a TVS-CMS over a cone $P$ in a metrizable locally convex topological vector space $(E,S)$, $S=\{q_k:k=1,2,...\}$ the set of seminorms generating the topology of $E$. Then

 (i) the function $d_S:X\times X\rightarrow\mathbb{R}$ defined by
$$d_S(x,y)=\inf\{h(u):p(x,y)\leq u,~u \in P\},$$ where $h(u)=\sum_{k=1}^\infty \frac{1}{2^k} \frac{q_k(u)}{1+q_k(u)},$ is a metric on $X$.

(ii) the metric topology $\tau_S$ induced by $d_S$ is finer than the cone metric topology $\tau_p$ induced  by $p$. In this case $d_S(x,y)=\inf\{h(u):p(x,y)\leq u,~u \in P\},$

(iii) If $E$ is a real Banach space (i.e $(X,p)$ is a CMS), then $(X,p)$ and $(X,d_S)$ are topologically isomorphic.
\end{theorem}
\begin{proof}
(i) Clearly $0\leq d_S(x,y)< \infty$ and $d_S(x,y)=d_S(y,x)$ for all $x,y \in X$. If $ d_S(x,y)=0$ then there exists $\{u_n\} \in E$ such that $p(x,y)\leq u_n$ and $h(u_n)<\frac{1}{n}$ for all $n \in \mathbb{N}$. Hence, $q_k(u_n)\rightarrow 0$ as $n\rightarrow \infty$ for all $k =1,2,...$. Let $c\gg 0$ be given then find $\eta > 0$ and $k_0 \in \mathbb{N}$ such that $q_{k_0}(b)< \eta$ implies $b\ll c$. Since $q_{k_0}(u_n)\rightarrow 0$ as $n\rightarrow \infty$, find $n_0$ such that $q_{k_0}(u_n)< \eta$ for all $n\geq n_0$. Hence, $p(x,y)\leq u_n\ll c$. Therefore, $p(x,y) \in P \cap -P$ which implies that $p(x,y)=0$ and so $x=y$. Now if $x=y$ then $p(x,y)=0$, which implies that $d_S(x,y)\leq h(u)$ for all $u \geq 0$ and hence $0\leq d_S(x,y)\leq h(0)=0$ and so $d_S(x,y)=0$.

To prove the triangle inequality, for $x,y,z \in X$ we have,
$$\forall \varepsilon > 0 ~~ \exists u_1: ~h(u_1)< d_S(x,z)+\varepsilon,~~p(x,z)\leq u_1,$$
and
$$\forall \varepsilon > 0 ~~ \exists u_2: ~h(u_2)< d_S(z,y)+\varepsilon,~~p(z,y)\leq u_2,$$
But $p(x,y)\leq p(x,z)+p(z,y)\leq u_1+u_2$ implies that
$$d_S(x,y)\leq h(u_1+u_2)\leq h(u_1)+h(u_2)\leq d_S(x,z)+d_S(z,y)+2\varepsilon.$$ Since $\varepsilon > 0$ was arbitrary then we have $d_S(x,y) \leq d_S(x,z)+d_S(z,y)$.

(ii) We show that if $x_n\rightarrow x$ in $(X,d_S)$ then  $x_n\rightarrow x$ in $(X,p)$. First from the definition of $d_S$ we can find $\{u_n\}$ in $E$ such that $$h(u_n)< d_S(x_n,x)+\frac{1}{n},~~p(x_n,x)\leq u_n.$$ Now, since $h(u_n)\rightarrow 0$ as $n\rightarrow \infty$ then $u_n\rightarrow 0$ in $(E,S)$. If $c\gg 0$ is given, find $n_0$ such that $u_n \ll c$ for all $n \geq n_0$. This implies that $p(x_n,x)\ll c$ for all $n \geq n_0$. That is  $x_n\rightarrow x$ in $(X,p)$.

(iii) Actually, this part is true even $E$ is a seminormed space. Let $\varepsilon > 0$ be given and choose $c \in E$ such that $c\gg 0$ and $\|c\|< \varepsilon$. This is possible since $int(P)$ is nonempty. Namely, let $c=\frac{\varepsilon c_0}{2\|c_0\|}$. Now, find $n_0$ such that $p(x_n,x)\ll c$ for all $n\geq n_0$. Hence, $d_S(x_n,x)\leq \|c\|< \varepsilon$  for all $n\geq n_0$.
\end{proof}

Note that to prove  that (iii) above is true in a certain metrizable locally convex space, we need to show that for every $\varepsilon > 0$ there exists $c\gg 0$ such that $h(c)<\varepsilon$.
\begin{example} \label{cone norm space}
Let $X=\omega:$ the vector space of all real sequences, $E=\omega$ and $P=\{x=\{x_i\}:x_i\geq 0, \texttt{for all i}\}$. Then $\omega$ is a complete metrizable locally convex space (Fr\'{e}chet space) when its topology is generated by the seminorms $\{q_k: k=1,2,3,..\}$ where $q_k(x)= \sum_{i=1}^k |x_i|$, $x=\{x_i\}\in \omega$. Clearly $P$ is a normal cone in $E$ but $int(P)=\Phi$ \cite{Ali}. In this connection, it is known that a Fr\'{e}chet space admits a continuous norm if and only if it does not contain $\omega$ as a complemented subspace. On $X$ define $d(x,y)= \{|x_i-y_i|\}_i,$ $x=\{x_i\}, ~y=\{y_i\} \in X$. Then $(X,d)$  is a TVS-cone metric space over $E$. In $E$, we can define $h$ naturally by $$h(u)=\sum_{k=1}^\infty \frac{1}{2^k} \frac{|u_k|}{1+|u_k|},~~u=\{u_k\} \in E.$$ It can be shown that for every $\varepsilon > 0$ there exists $c\gg 0$ such that $h(c)<\varepsilon$. Actually, for $ \varepsilon > 0$, choose $c=(\varepsilon,0,0,...)\gg 0$, then $h(c)< \varepsilon$.
 \end{example}

\begin{definition} \label{bdd}
A subset $A$ of a TVS-CMS $(X,p)$ is called bounded above if there exists $c \in E,~c\gg 0$ such that $p(x,y)\leq c.$ for all $x,y \in A$, and is called bounded if $\delta (A)=\sup\{d(x,y): x, y \in A\}$ exists in $E$. If the supremum does not exist, we say that $A$ is unbounded.
\end{definition}
\begin{proposition} \label{bounded}
Let $(X,p)$ be TVS-CMS over a strongly minihedral normal cone $P$ in $(E,S)$ and $A \subset X$. Then, $A$ is bounded if and only if $\delta_q(A)=\sup\{q(p(x,y)):x,y \in A\}< \infty$ for all $q \in S$ (i.e $\{p(x,y):x,y \in A\}$ is bounded in $(E,S)$).
\end{proposition}
\begin{proof}
Assume $A$ is bounded and hence find $c\gg 0$ such that $p(x,y)\leq c$ for all $x,y \in A$. Let $q \in S$ be given, then by normality $\delta_q(A)=\sup \{q(p(x,y)):x,y \in A\}\leq q(c)<\infty $.

Conversely, assume $\delta_q(A)<\infty$ for all $q \in S$ and fix some $c_1\gg 0$. Then find $q_0 \in S$ and $\eta >0$ such that $q_0(b)< \eta$ implies $b \ll c_1$. For each $x,y \in A$, let $c_{x,y}=\frac{\eta p(x,y)}{2q_0(p(x,y))}$. Then $q_0(c_{x,y})= \frac{\eta}{2}< \eta$ implies $c_{x,y}\ll c_1 $ and so $\frac{2 q_0(p(x,y))}{\eta}c_1- p(x,y) \in int(P)$. Therefore, $p(x,y)\ll \frac{2 q_0(p(x,y))}{\eta}c_1 \leq \frac{2\delta_{q_0}}{\eta}c_1=c$. Since $P$ is strongly minihedral then  $A$ is bounded.
\end{proof}
In fact, Corollary \ref{bounded} above shows that if the cone is normal then $A$ is bounded above if and only if $\{p(x,y):x,y \in A\}$ is bounded in $(E,S)$.
\section{Fixed Point Theory Remarks}

\begin{theorem} \label{FP}
Let $(X,p)$ be a TVS-CMS and $(X,d_p)$ its correspondent metric space, $d_p=\xi_e \circ p$. Assume $T:X \rightarrow X$ is a mapping such that $p(Tx,Ty) \leq \varphi (p(x,y))$ for all $x,y \in X$, where $\varphi: P\rightarrow P$ is increasing and $\varphi(re)\leq r \varphi(e)$ for every real number  $r\geq 0$. Then, there exists linear, increasing and continuous mapping $\phi: \mathbb{R}^+\rightarrow \mathbb{R}^+$ such that $d_p(Tx,Ty)\leq \phi (d_p(x,y))$, for all $x,y \in X$.
\end{theorem}
\begin{proof}
Define $\phi(t)=\xi_e(\varphi(e))t$ for every $t \geq 0$. Then clearly, $\phi$ is linear, continuous, increasing, $\phi(0)=0$ and $\phi(\xi_e(e))=\phi(1)=\xi_e(\varphi(e))$. To prove our claim, let $\varepsilon > 0$ be given and find $r \in \mathbb{R}$ such that $0 \leq d_p(x,y)<r < d_p(x,y)+\varepsilon$, where $p(x,y) \in re-P$ (or $p(x,y)\leq re$). Now, by  Lemma \ref{lemma_scalarization} and that $\varphi$ is increasing, we have
$$d_p(Tx,Ty)=\xi_e(p(Tx,Ty))\leq$$
\begin{equation} \label{incr}
\xi_e(\varphi(p(x,y)))\leq \xi_e(\varphi(re))\leq r \xi_e(\varphi(e))=r \phi(1)=\phi(r)\leq \phi(d_p(x,y))+ \phi(\varepsilon)
\end{equation}
Since $\varepsilon$ is arbitrary and $\phi(0)=0$ we conclude that $d_p(Tx,Ty)\leq \phi(d_p(x,y))$.
\end{proof}
\begin{theorem} \label{Boyd} \cite{Wong}
Let $(X,d)$ be a complete metric space and let $T:X\rightarrow X$ be a mapping satisfying
$$d(Tx,Ty)\leq \varphi (d(x,y)),~~\texttt{for all} ~x,y \in X,$$ where $\phi: \mathbb{R}^+\rightarrow {R}^+$ is uppersemicontinuous from the right, satisfying $\phi(t)<t$ for $t>0$. Then, $T$ has a unique fixed point, say $u$, and $T^nx\rightarrow u$ as $n\rightarrow \infty$ for all $x \in X$.
\end{theorem}
\begin{proposition} \label{moregeneral}
Let $(X,p)$ be a complete TVS-CMS and let $T:X\rightarrow X$ be a mapping satisfying
$$p(Tx,Ty)\leq \varphi (p(x,y)),~~\texttt{for all} ~x,y \in X,$$
where $\varphi: P\rightarrow P$ is increasing,  $\varphi(re)\leq r \varphi(e)$ for every real number  $r\geq 0$, and $\xi_e(\varphi(e))< 1$. Then, $T$ has a unique fixed point, say $u$, and $T^nx\rightarrow u$ as $n\rightarrow \infty$ for all $x \in X$.
\end{proposition}
The proof follows by Theorem \ref{Boyd} and Theorem \ref{FP}.

\begin{theorem} \label{FP2}
Let $(X,p)$ be a TVS-CMS and $(X,d_p)$ its correspondent metric space, $d_p=\xi_e \circ p$. Assume $T:X \rightarrow X$ is a mapping such that $p(Tx,Ty) \leq \varphi (p(x,y))$ for all $x,y \in X$, where $\varphi: P\rightarrow P$ is increasing, continuous, $\varphi(0)=0$ and $\varphi(re)\ll re$ for every real number  $r > 0$. Then, there exists a continuous increasing mapping $\phi: \mathbb{R}^+\rightarrow \mathbb{R}^+$ such that $\phi(t)< t$ for $t>0$ and $d_p(Tx,Ty)\leq \phi (d_p(x,y))$, for all $x,y \in X$.
\end{theorem}
\begin{proof}
Define $\phi(r)=\xi_e(\varphi(re))$ for every $r \geq 0$. Then clearly,  by  Lemma \ref{lemma_scalarization}, $\phi$ is  continuous, increasing, $\phi(0)=0$ and $\phi(r+s)\leq \phi(r)+ \phi(s)$ for all $r,s>0$. To prove our claim, let $\varepsilon > 0$ be given and find $r \in \mathbb{R}$ such that $0 \leq d_p(x,y)<r < d_p(x,y)+\varepsilon$, where $p(x,y) \in re-P$ (or $p(x,y)\leq re$). Now, by  Lemma \ref{lemma_scalarization} and that $\varphi$ is increasing, we have
$$d_p(Tx,Ty)=\xi_e(p(Tx,Ty))\leq$$
\begin{equation} \label{incr}
\xi_e(\varphi(p(x,y)))\leq \xi_e(\varphi(re))=\phi(r)\leq \phi(d_p(x,y))+ \phi(\varepsilon)
\end{equation}
Since $\varepsilon$ is arbitrary and $\phi(0)=0$ we conclude that $d_p(Tx,Ty)\leq \phi(d_p(x,y))$.
\end{proof}
\begin{proposition} \label{moregeneral}
Let $(X,p)$ be a complete TVS-CMS and let $T:X\rightarrow X$ be a mapping satisfying
$$p(Tx,Ty)\leq \varphi (p(x,y)),~~\texttt{for all} ~x,y \in X,$$
where $\varphi: P\rightarrow P$ is increasing uppersemicontinuous from right,  $\varphi(re)\ll r e$ for every real number  $r>0$. Then, $T$ has a unique fixed point, say $u$, and $T^nx\rightarrow u$ as $n\rightarrow \infty$ for all $x \in X$.
\end{proposition}
The proof follows by Theorem \ref{Boyd} and Theorem \ref{FP2}.

\begin{remark}
Theorem \ref{FP} above can also be used to generalize some quasi-contraction type fixed point theorems from metric spaces to TVS-cone metric spaces.
\end{remark}

\begin{remark}
Even it is possible to find  metric spaces which are topologically isomorphic to the original TVS-CMS, it is still interesting to ask or comment on the following:

(i) Is it possible to give rise to nonlinear contraction type conditions with $\varphi$ not neccesirly increasing,  by the nonlinear scalarization from TVS-CMS to usual metric spaces and thus provides trivial proofs to the TVS-CMS counterpart of the results such as in \cite{Wong} and without any regularity assumption. If it is not the case, can we adapt the proof in \cite{Wong} to regular complete TVS-CMS? That is can we remove the monotonicity of $\varphi$ from Theorem \ref{FP2}.

(ii) There are many properties which are not topological thus they can not pass from $(X,p)$ to $(X,d_p)$.

(iii) Proving a completion theorem for TVS-CMS is still interesting. Actually, even  $(X,p)$ and $(X,d_p)$ are topologically isomorphic we can't say they are isometric. See the discussion in the previous section.

(iv) Is it trivial to generalize all Meir-Keeler type fixed point theorems to TVS-CMS and just by applying the nonlinear scalarization function?

(v)For a closed and convex subset C of a  Banach space and  $T : C \rightarrow C$ satisfies the conditions $0\leq s+|a|-2ab< 2(a+b)$
$4a d(Tx,Ty)+b (d(x,Tx)+d(y,Ty))\leq)\leq s d(x,y)$, $d(x,y)=\|x-y\|$,  T has at least one fixed point (for example see \cite{Erdal}). The nonlinear scalarization $\xi_e$ can not be applied to generalize such   kind of theorems to TVS-CMS. However, in the proof of such cases Lemma \ref{th2} will be crucial.

(vi) We may conclude that the nonlinear scalarization does not work if extra normality or regularity assumptions on the cone $P$ are requested. For example, the following theorem which has been proved in \cite{Choud}:
\begin{theorem} \label{weak contraction}\cite{Choud}
Let $(X,d)$ be a complete cone metric space with regular cone $P$ such that $d(x,y) \in int P$ for each $x, y \in P$ with $x\neq y$. Let $T:X\rightarrow X$ be a mapping satisfying the inequality
\begin{equation} \label{w1}
d(Tx,Ty)\leq d(x,y)-\varphi(d(x,y)),~~\emph{for} ~~x,y \in X
\end{equation}
where $int P\cup \{0\}\rightarrow int P\cup \{0\}$ is a continuous and monotone function with

(i) $\varphi(t)=0$ if and only if $t=0$

(ii) $\varphi (t)\ll t$ for $t\gg 0$,

(ii) either $\varphi(t)\leq d(x,y)$ or $d(x,y)\leq \varphi(t)$ for $t \in int P\cup \{0\}$ and $x,y \in X$.

Then $T$ has a unique fixed point in $X$.

It is clear that the properties of the nonlinear scalarization are not enough to rise the condition (\ref{w1}) to usual metric space.
\end{theorem}
\end{remark}


\end{document}